\documentclass[12pt]{article}
\usepackage{amsmath,amsthm,amssymb,amsfonts}
\usepackage[numbers,sort&compress]{natbib}
\usepackage{color,colordvi}
\usepackage{soul}
\usepackage{fullpage}

\usepackage[letterpaper,top=2cm,bottom=2cm,left=3cm,right=3cm,marginparwidth=2cm]{geometry}

\usepackage{amsmath,amsfonts,amsthm,mathtools,bm}
\usepackage{graphicx}
\usepackage[colorlinks=true, allcolors=blue]{hyperref}
\usepackage[capitalise]{cleveref}
\usepackage{standalone}
\usepackage{indentfirst}
\usepackage{authblk}
\usepackage{diagbox}
\usepackage{changes}
\usepackage{enumitem}
\usepackage{bbm}
\usepackage{orcidlink}

\newtheorem{theorem}{Theorem}[section]

\newtheorem{lemma}[theorem]{Lemma}
\newtheorem{corollary}[theorem]{Corollary}

\newtheorem{remark}[theorem]{Remark}

\newtheorem*{fact*}{Fact}

\newtheorem*{claim*}{Claim}

\newcommand{\T}{\mathbb T}
\newcommand{\R}{\mathbb R}
\newcommand{\Z}{\mathbb Z}
\newcommand{\supp}{\operatorname{supp}}
\newcommand{\Sch}{\mathcal S}

\title{Sharp dispersive estimates for wave equation on  lattice graphs with high dimension}
\author{Zhe You$^a$\footnote{email: y30231280@mail.ecust.edu.cn}~\orcidlink{0009-0006-9769-5372} }
\author{Xiongfeng Zhan$^{b}$\footnote{email: zhanxfmath@163.com }}
\affil{$^a$School of Mathematics, East China University of Science and Technology, Shanghai 200237, China.}
\affil{$^b$School of Mathematics and Statistics, Beijing Jiaotong University, Beijing 100044, China.}

\date{}

\begin{document}
	\maketitle
	
	\begin{abstract}
		Schultz [Comm.~Pure~Appl.~Math., 1998]  first established dispersive estimates for the fundamental solution of the discrete wave equation on lattice graphs $\Z^d$ for $d=2,3$, which can be seen as a discrete analogue of the classical dispersive estimate for the Euclidean wave equation.
		His result was extended by Bi, Cheng and Hua to  $d=4$ and $5$.
		In this paper, we give complete answers to all the remaining  dimensions, including $d=5$.
		Our estimate is sharp for every $d\geq 5$. 
		Moreover, the proof remains valid for the upper bound of the case for $d=4$.
	\end{abstract}
	
	
	
	\section{Introduction}\label{sec:1}
	Discrete analogs of the partial differential equations on graphs have attracted much attention in recent years. 
	We refer to the monographs \cite{Bar17,Gri18} for this topic. 
	As an important equation in classical PDEs, the wave equation is also a fundamental object of study in  discrete settings.
	The discrete wave equation on lattice graphs, which can be regarded as a discrete analogue of the Euclidean case, provides a natural way to approximate real physical laws numerically. 
	In physics, they can be used to describe the vibrations of atoms in crystals \cite{Dove93,FH10,Kev11} via monotonic chains. As for general graphs, one can see \cite{FT04,HH24,LX19,LX22} for results on discrete wave equations. 
	
	In the study of nonlinear problems, establishing dispersive inequalities for the underlying linear dispersive equation usually comes first.
	It is well known that the decay rate of the solution to wave equation on Euclidean space $\R^d$ is of order $|t|^{-\frac{d-1}{2}}$, see for example, \cite{GV95,KT98,Str77}.  
	However, it behaves a little bit differently in discrete cases. 
	In this paper, we focus on the discrete version of this problem. 
	Consider the discrete wave equation on the lattice graph $\Z^d$
	\begin{equation*}
		\begin{cases}
			\partial_t^2 u(x,t)-\Delta u(x,t)=F(x,t),\\[4pt]
			u(x,0)=f_1(x),\qquad
			\partial_t u(x,0)=f_2(x),
			\qquad (x,t)\in\mathbb{Z}^d\times\mathbb{R}.
		\end{cases}
	\end{equation*}
	Here the discrete Laplacian $\Delta$ is defined by
	\begin{equation*}
		\Delta f(x)=\sum_{j=1}^d\bigl(f(x+e_j)+f(x-e_j)-2f(x)\bigr)
	\end{equation*}
	with $\{e_j\}_{j=1}^d$ the standard basis of the lattice $\Z^d$.
	The fundamental solution of this equation is given by
	\begin{equation*}\label{eq:kernel}
		G_d(x,t)=\frac{1}{(2\pi)^d}\int_{\T^d}e^{ix\cdot\xi}
		\frac{\sin(t\omega(\xi))}{\omega(\xi)}\,d\xi,
		\quad
		\omega(\xi)=\left(\sum_{j=1}^d(2-2\cos\xi_j)\right)^{1/2},
	\end{equation*}
	where the $d$-dimensional torus $\T^d$ is the cube $[-\pi,\pi]^d$ with opposite sides identified. 
	
	Schultz~\cite{Sch98} proved the sharp estimates of $G_d(x,t)$ in dimensions $2$ and $3$.  
	Bi, Cheng and Hua~\cite{BCH24,BCH26} subsequently proved the sharp $4$-dimensional bound $(1+|t|)^{-3/2}\log(2+|t|)$ and  $5$-dimensional bound $(1+|t|)^{-11/6}$.  
	Their proofs involved the Newton polyhedron.
	They also noticed that the Newton polyhedron can be used not only to interpret the decay
	rates for $d = 2,3,4$, but also to study the most degenerate case for all odd $d \geq 3$.
	However, their approach becomes increasingly complicated in higher dimension.

	The main result of this paper is as follows.
	\begin{theorem}\label{thm:main}
		Let $d\geq5$ and $\sigma_d=\frac{2d+1}{6}$.
		There exist constants $C_d,c_d,T_d>0$ such that
		\begin{equation}\label{eq:uppermain}
			|G_d(x,t)|\leq C_d(1+|t|)^{-\sigma_d},
			\quad (x,t)\in\Z^d\times\R,
		\end{equation}
		and
		\begin{equation*}
			\sup_{x\in\Z^d}|G_d(x,t)|\geq c_d(1+|t|)^{-\sigma_d},
			\quad |t|\geq T_d.
		\end{equation*}
	\end{theorem}
	
	Our proof strategy is as follows. 
	First, we decompose the kernel into
	low- and high-frequency parts, and apply the estimates of Schultz to the low-frequency part and the region away from the propagation
	cone. 
	Then, by applying Fourier inversion in a suitable scalar variable, we represent the high-frequency part as a one-parameter superposition of products of one-dimensional oscillatory integrals.
	Uniform estimates for these one-dimensional factors, together with a product-integrability argument, yield the desired upper bound. 
	Finally, to prove the sharpness of the decay rate, we localize near the most degenerate frequency and combine an anisotropic rescaling with the Poisson summation formula to obtain the desired lower bound.
	
	It is worth mentioning that our method also works for the $4$-dimensional bound $(1+|t|)^{-3/2}\log(2+|t|)$ in~\cite{BCH24}, see Corollary~\ref{corollary} and Remark~\ref{remark}.
	
	Two other related models are the discrete Schr\"odinger equation
	\begin{equation*}
		i\partial_tu(x,t)+\Delta u(x,t)=0,
	\end{equation*}
	and the discrete Klein-Gordon equation
	\begin{equation*}
		\partial_t^2u(x,t)-\Delta u(x,t)+m_*^2u(x,t)=0,
	\end{equation*}
	where $m_*>0$ is the mass parameter.
	The discrete Schr\"odinger equation has already been well studied in \cite{SK05}, and the sharp decay rate of its solutions is $|t|^{-\frac{d}{3}}$ on $\Z^d$.  
	Because of the special form of its Green’s function, one can separate variables to reduce the problem to the case $d = 1$. 
	The same argument fails in the case of the discrete Klein–Gordon equation, which makes the situation more difficult.
	It has been studied in \cite{SK05} for $d=1$, \cite{BG17} for $d=2$, and \cite{CI21} for $d=2,3,4$. 
	Moreover, the lack of smoothness of the phase and amplitude at the origin makes the situation for the discrete wave equation harder than that for the discrete Klein–Gordon equation.

	
	A byproduct of this paper is a sharp dispersive estimate for the discrete Klein–Gordon equation in all dimensions $d\geq 5$, which is the same estimate as the one for the discrete wave equation.
	The proof is also the same as that of the discrete wave equation and can be proved by essentially the same argument.
	In fact, it is even simpler since it does not involve the stationary case.
	We do not give the details in this paper.

	\section{Preliminaries}\label{sec:2}
	
	In this section, we recall some tools that will be used later. 
	First we give some notations for convenience.  
	Given two expressions $A$ and $B$, we write $A = O_{p}(B)$ to express that $A \leq C_p B$ for some constant $C_p > 0$ which  may depend on $p$. 
	If the constant $C_p$ is universal, then we omit $p$ in the subscript.
	For convenience, we also write $A \lesssim_p B$ and $A \lesssim B$  as a synonym for $A=O_p(B)$ and $A=O(B)$, respectively.
	We write $A\asymp B$ when $A\lesssim B$ and $B\lesssim A$.

	\subsection{Fourier transform}
	We refer to \cite[Chapter~2]{Gra14} or \cite[Chapter~6]{SS03} for the basic theory of Fourier transform. 
	
	A \emph{multi-index} is a $d$-tuple $\alpha=(\alpha_1,\dots,\alpha_d)$ of nonnegative integers. 
	Let $|\alpha|=\alpha_1+\cdots+\alpha_d$.
	For $x=(x_1,\dots,x_d)\in\R^d$, the monomial $x^\alpha$ is defined by
	$x^\alpha=x_1^{\alpha_1}\cdots x_d^{\alpha_d}$, and the differential operator $\partial^\alpha$ is defined by
	$\partial^\alpha=\frac{\partial^{|\alpha|}}{\partial x_1^{\alpha_1}\cdots\partial x_d^{\alpha_d}}$.
	
	The \emph{Schwartz space} $\Sch(\R^d)$ is the space of all smooth functions $f:\R^d\to\mathbb C$ such that
	\begin{equation*}
		\sup_{x\in\R^d}\bigl|x^\alpha\partial^\beta f(x)\bigr|< +\infty
	\end{equation*}
	for every pair of multi-indices $\alpha,\beta$. 
	In other words, $f$ and all of its derivatives are required to be rapidly decreasing.  
	In particular, for every integer $N\geq0$, there is a constant $C_{N}$ such that
	\begin{equation}\label{eq:rapidly decreasing}
		|f(x)|\leq C_N(1+|x|)^{-N}
	\end{equation}
	for every $x\in \R^d$.
	Let $C_c^\infty(U)$ be the space of all smooth functions on an open set $U\subseteq\R^d$ whose support is compact. 
	It is straightforward to verify that $C_c^\infty(\R^d)\subset \Sch(\R^d)$.
	
	For a function $f\in \Sch(\R^d)$, the \emph{Fourier transform} of $f$ is defined by
	\begin{equation*}
		\widehat f(\eta)=\int_{\R^d}e^{-i\eta\cdot x}f(x)\,d x.
	\end{equation*}
	The Fourier transform is a bijective map of $\Sch(\R^d)$ to itself, whose inversion formula is 
	\begin{equation*}
		f(x)=\frac{1}{(2\pi)^d}\int_{\R^d}e^{i\eta\cdot x}\widehat f(\eta)\,d\eta.
	\end{equation*}
	
	We also need the Poisson summation formula as follows; see also~\cite[Chapter~3, Theorem 3.2.8]{Gra14}.
	
	\begin{lemma}[{\cite[Theorem 3.2.8]{Gra14}}] \label{lem:Poisson formula}
		Let $f\in\Sch(\R^d)$. 
		Then
		$\sum\limits_{n\in\Z^d}f(n)=\sum\limits_{k\in\Z^d}\widehat f(2\pi k)$.
	\end{lemma}
	
	\subsection{Oscillatory integrals}
	In this subsection, we recall some classical results about oscillatory integrals.
	We begin with a lemma which records conditions under which repeated integration by parts shows that an oscillatory integral is very small.
	\begin{lemma}[{\cite[Lemma 8.1]{BKY13}}]\label{lem:BKY}
		Let $J$ be a fixed bounded closed interval, let 
		$a\in C^\infty(\R)$ with support on $J$, and let $\Phi$ be real-valued and smooth on $J$.
		Suppose that there exist some variables $Y\geq1$ and $R>0$ such that for every $k\geq0$, $\|a^{(k)}\|_{L^\infty(J)}\lesssim_k 1$, and for every $k\geq2$,
		$|\Phi'(s)|\geq R$ and $\|\Phi^{(k)}\|_{L^\infty(J)}\lesssim_kY$.
		Then, for every $A>0$,
		$$
		\left|\int_{\R} e^{i\Phi(s)}a(s)\,ds\right|
		\lesssim_A
		\left(\frac{R}{\sqrt Y}\right)^{-A}+R^{-A}.
		$$
	\end{lemma}
	
	A zero of $\Phi'$ is called a \emph{stationary point}. 
	Near such a point the preceding estimate cannot be applied, and we use the following form of the Van der Corput lemma.
	
	\begin{lemma}[{\cite[p. 334]{Stein93}}]\label{lem:van der Corput}
		Let $J$ be a bounded open interval.
		Let $a\in C^1(J)$, and let
		$\Phi$ be real-valued and smooth on $J$.
		\begin{enumerate}[label=(\roman*)]
			\item If $\Phi'$ is monotone and $|\Phi'|\geq\Lambda>0$ on $J$,
			then
			\[
			\left|\int_J e^{i\Phi(s)}a(s)\,d s\right|
			\lesssim \Lambda^{-1}
			\bigl(\|a\|_{L^\infty(J)}+\|a'\|_{L^1(J)}\bigr).
			\]
			\item If $|\Phi^{(k)}|\geq\Lambda>0$ on $J$ for some $k>1$, then
			\[
			\left|\int_J e^{i\Phi(s)}a(s)\,d s\right|
			\lesssim_k\Lambda^{-1/k}
			\bigl(\|a\|_{L^\infty(J)}+\|a'\|_{L^1(J)}\bigr).
			\]
		\end{enumerate}
	\end{lemma}
	
	\section{Cutoff and one-dimensional factorization}\label{sec:3}
	Schultz~\cite{{Sch98}} introduced an oscillatory integral to study dispersive estimates for the wave equation on lattice graphs $\Z^d$ for $d=2,3$, which has the form
	\begin{equation*}\label{eq:K_d}
		K_d(x,t)=\frac{1}{(2\pi)^d}\int_{\T^d}e^{ix\cdot\xi-it\omega(\xi)}\frac{d\xi}{\omega(\xi)}.
	\end{equation*}
	The elementary identity $\sin z=(e^{iz}-e^{-iz})/(2i)$ and the symmetry $\omega(-\xi)=\omega(\xi)$ give
	$G_d(x,t)=-\operatorname{Im}K_d(x,t)$.
	Thus an upper bound for $|K_d|$ implies the same upper bound for $|G_d|$.
	
	We now introduce the cutoff that will be used throughout
	this section. 
	A \emph{cutoff} $\chi$ is a smooth function $\chi$  supported near the origin.
	For a suitable cutoff, we can decompose $K_d(x,t)$ as
	$	K_d(x,t)=K_d^{\mathrm{low}}(x,t)+K_d^{\mathrm{hi}}(x,t)$,
	where
	\begin{equation*}
		K_d^{\mathrm{low}}(x,t)
		=
		\frac{1}{(2\pi)^d}
		\int_{\T^d}
		e^{ix\cdot\xi-it\omega(\xi)}
		\frac{\chi(\xi)}{\omega(\xi)}
		\,d\xi
	\end{equation*}
	and
	\begin{equation*}
		K_d^{\mathrm{hi}}(x,t)
		=
		\frac{1}{(2\pi)^d}
		\int_{\T^d}
		e^{ix\cdot\xi-it\omega(\xi)}
		\frac{1-\chi(\xi)}{\omega(\xi)}
		\,d\xi.
	\end{equation*}
	Then we recall the following result of Schultz~\cite[Propositions~2.1--2.3 and 3.10]{Sch98}; see also \cite[Lemma~3.3 and equation~(25)]{BCH24}.  
	
	\begin{lemma}[{\cite{Sch98}}]\label{lem:schultz}
		Let $d\geq 4$.  There exist constants $c_d\in(0,1)$ such that:
		\begin{enumerate}[label=(\roman*)]
			\item if $|x|\geq c_d|t|$, then $|G_d(x,t)|\lesssim_d(1+|t|)^{-d/2}$;
			\item if $|x|\leq c_d|t|$, then $|K_d^{\mathrm{low}}(x,t)|\lesssim_d|t|^{-(d-1)}$ for $|t|\geq1$
			with a cutoff  $\chi$ on $\T^d$.
		\end{enumerate}
	\end{lemma}
	Now we need to find a suitable cutoff. 
	Set $S(\xi)=-\sum_{j=1}^d\cos\xi_j$.
	Then $ d+S(\xi) =2\sum_{j=1}^{d}\sin^2\frac{\xi_j}{2}$.
	For $\xi\in[-\pi,\pi]^d$, we have
	$ \frac{2}{\pi^2}|\xi|^2\le d+S(\xi)\le \frac12|\xi|^2$.
	For sufficiently small $\delta>0$,
	let $\chi_0$ be a cutoff such that $\chi_0(r)=1$ for $|r|\leq\delta$ and $\chi_0(r)=0 $ for $|r|\geq2\delta$.
	Define $\chi_\delta(\xi)\coloneqq\chi_0(d+S(\xi))$.
	Note that $\operatorname{supp}\chi_\delta
	\subset
	B_{\mathbb T^d}
	\left(0,\pi\sqrt{\delta}\right)$.
	Thus $\chi_\delta$ is still a cutoff on $\T^d$ supported in a sufficiently small neighborhood of the origin.

	In this paper, we always use such a cutoff; that is,
	\begin{equation}\label{eq:K low}
		K_d^{\mathrm{low}}(x,t)
		=
		\frac{1}{(2\pi)^d}
		\int_{\T^d}
		e^{ix\cdot\xi-it\omega(\xi)}
		\frac{\chi_0(d+S(\xi))}{\omega(\xi)}
		\,d\xi
	\end{equation}
	and
	\begin{equation}\label{eq:K high}
		K_d^{\mathrm{hi}}(x,t)
		=
		\frac{1}{(2\pi)^d}
		\int_{\T^d}
		e^{ix\cdot\xi-it\omega(\xi)}
		\frac{1-\chi_0(d+S(\xi))}{\omega(\xi)}
		\,d\xi.
	\end{equation}
	The estimate of~\eqref{eq:K low} is exactly covered by Lemma~\ref{lem:schultz} (ii).  
	Therefore, it suffices to estimate~\eqref{eq:K high}. 
	Define a function $a$ on $[-d,d]$ by setting $a(-d)=0$ and, for $s\in(-d,d]$,
	\begin{equation}\label{eq:as}
		a(s)=\frac{1-\chi_0(d+s)}{\sqrt{2d+2s}}.
	\end{equation}
	The numerator in \eqref{eq:as} vanishes on a neighborhood of $s=-d$, precisely where  the denominator would be singular. Hence $a$ is well-defined and smooth on $[-d,d]$.  
	We extend it smoothly to a compactly supported function in $C_c^\infty((-d,\infty))$, and then extend it by zero to $\R$.
	We also denote this new function by the same notation $a$.
	Hence \eqref{eq:K high} becomes
	\begin{equation}\label{eq:K high2}
		K_d^{\mathrm{hi}}(x,t)
		=\frac{1}{(2\pi)^d}\int_{\T^d}
		e^{ix\cdot\xi-it\omega(\xi)}a(S(\xi))\,d\xi.
	\end{equation}
	
	For $m\in\Z$, let $B_m(\eta)=\int_{-\pi}^{\pi}e^{im\theta-i\eta\cos\theta}\,d\theta$.
	The next lemma is key to the upper bound estimate about one-dimensional factorization. 
	\begin{lemma}\label{lem:factorization}
		For $t\in \R$, let $F_t$ be the function defined by $F_t(s)=a(s)e^{-it \sqrt{2d+2s}}$.
		Then
		\begin{equation}\label{eq:factor}
			K_d^{\mathrm{hi}}(x,t)=\frac{1}{(2\pi)^{d+1}}
			\int_{\R}\widehat F_t(\eta)\prod_{j=1}^dB_{x_j}(\eta)\,d\eta,
		\end{equation}
		and the integral is absolutely convergent.
	\end{lemma}
	
	\begin{proof}
		By~\eqref{eq:K high2} and the definition of $F_t$, we obtain
		\begin{align*}
			K_d^{\mathrm{hi}}(x,t)
			&=\frac{1}{(2\pi)^d}\int_{\T^d}
			e^{ix\cdot\xi-it\omega(\xi)}a(S(\xi))\,d\xi\\
			&=\frac{1}{(2\pi)^d}\int_{\T^d}e^{ix\cdot\xi}F_t(S(\xi))\,d\xi.
		\end{align*}
		Since $F_t\in C_c^\infty(\R)\subset \Sch(\R)$, the Fourier inversion formula gives
		$$ F_t(S(\xi))=\frac{1}{2\pi}\int_{\R}\widehat F_t(\eta)e^{i\eta S(\xi)}\,d\eta.
		$$
		Since $S(\xi)=-\sum_{j=1}^d\cos\xi_j$, we have $e^{i\eta S(\xi)}
		=\prod_{j=1}^de^{-i\eta\cos\xi_j}$.
		Thus, separating $\xi_1,\cdots,\xi_d$ gives
		$$
		\int_{\T^d}\prod_{j=1}^d
		e^{ix_j\xi_j-i\eta\cos\xi_j}\,d\xi=
		\prod_{j=1}^d
		\int_{-\pi}^{\pi}e^{ix_j\theta-i\eta\cos\theta}\,d\theta=
		\prod_{j=1}^dB_{x_j}(\eta).
		$$
		By Fubini's theorem,~\eqref{eq:factor} follows immediately.
		Note that $\widehat F_t$ is also in $\Sch(\R)$. 
		Hence absolute convergence follows from \eqref{eq:rapidly decreasing} and the trivial bound $|B_m(\eta)|\leq 2\pi$.
	\end{proof}
	
	We now proceed to estimate $\widehat{F}_t$.    
	
	\begin{lemma}\label{lem:hat F}
		Fix the dimension $d$ and the cutoff $\chi_0$. 
		There exist constants $c_2>c_1>0$ such that
		$|\widehat F_t(\eta)|\lesssim|t|^{-1/2}$,
		for $|t|\geq1$ and $\eta\in\R$, and 
		\begin{equation}\label{eq:int hat F}
			\int_{\R\setminus E_t}|\widehat F_t(\eta)|\,d\eta
			\lesssim_N|t|^{-N} 
		\end{equation}
		for any $N\geq 0$, where $E_t=\{\eta\in\R:c_1|t|\leq|\eta|\leq c_2|t|\}$.
	\end{lemma}
	
	\begin{proof}
		By symmetry, we may assume that $t>0$. 
		By the Fourier transform, we have
		$$
		\widehat F_t(\eta)=\int_{\R} a(s)e^{-i(tg(s)+\eta s)}\,ds
		$$
		with $g(s)=\sqrt{2d+2s}$.
		On the fixed support of $a$, the point $s=-d$ is excluded.
		We also have $g'(s)=(2d+2s)^{-1/2}$ and $g''(s)=-(2d+2s)^{-3/2}$.
		Hence there are constants $0<m<M<+\infty$ and $c>0$ such that
		\begin{equation}\label{eq:g derivative}
			m\leq g'(s)\leq M\quad \text{and} \quad |g''(s)|\geq c
		\end{equation}
		on $\supp a$.  Hence the second derivative of  $tg(s)+\eta s$ is $t g''(s)$ and satisfies $|t g''(s)|\geq ct$ on $\supp a$. By Lemma~\ref{lem:van der Corput} (ii), $|\widehat F_t(\eta)|\lesssim|t|^{-1/2}$ holds.
		
		Then we prove~\eqref{eq:int hat F}.
		As a stationary point can occur only when $\eta=-tg'(s)$, necessarily we consider the  possible stationary  interval $-Mt\leq \eta\leq -mt$.
		We claim that $tg'(s)+\eta$ is uniformly bounded away from zero outside this interval.  
		Enlarge the possible stationary interval slightly as $I_t=[-(M+1)t,-(m/2)t]$.
		If $\eta>-(m/2)t$, then by \eqref{eq:g derivative}, we have $tg'(s)+\eta\geq mt+\eta>\frac{m}{2}t$.
		If $\eta<-(M+1)t$, then by~\eqref{eq:g derivative} again, we have $tg'(s)+\eta\leq Mt+\eta<-t$.
		Thus, we obtain $|tg'(s)+\eta|\geq c t $
		with some constant $c>0$ independent of $s,t$ and $\eta$ outside $I_t$. 
		We also require another lower bound for $|tg'(s)+\eta|$.  If $\eta>-mt$, then
		$$
		|tg'(s)+\eta|=tg'(s)+\eta\geq mt+\eta
		=\operatorname{dist}(\eta,[-Mt,-mt]).
		$$
		If $\eta<-Mt$, then
		$$
		|tg'(s)+\eta|=-(tg'(s)+\eta)\geq-\eta-Mt
		=\operatorname{dist}(\eta,[-Mt,-mt]).
		$$
		Combining two lower bounds, for $\eta\notin I_t$, we get
		\begin{equation*}
			|tg'(s)+\eta|\geq \frac{1}{2}\Bigl(ct+\operatorname{dist}(\eta,[-Mt,-mt])\Bigr)\gtrsim D(\eta,t),
		\end{equation*}
		where $D(\eta,t)=t+\operatorname{dist}(\eta,[-Mt,-mt])$.
		On the fixed support of $a$, all derivatives of $a$ are uniformly
		bounded, while every derivative of $tg(s)+\eta s$ of order
		at least two is $O(t)$.
		Therefore, we may apply Lemma~\ref{lem:BKY} with $Y\asymp t$, $R\asymp D(\eta,t)$, and $A=2L$.  
		Since
		$D(\eta,t)\geq t\geq1$ outside $I_t$, it follows that
		\begin{equation}\label{eq:F bound}
			|\widehat F_t(\eta)| \lesssim_L\left(\frac{D(\eta,t)}{\sqrt{t}}\right)^{-2L}
			+D(\eta,t)^{-2L}
			\lesssim_L D(\eta,t)^{-L}
		\end{equation}
		for every $L\geq1$ and $\eta\notin I_t$.
		Now choose $0<c_1<m/2$ and $c_2>M+1$.  
		Then the possible stationary interval $[-Mt,-mt]$ is contained in $E_t$, and every $\eta\notin E_t$ lies outside $I_t$.  
		By taking $L=N+2$ and integrating two sides of~\eqref{eq:F bound} over $\R\setminus E_t$, we obtain~\eqref{eq:int hat F}.
	\end{proof}
	
	\section{The upper bound}\label{sec:4}

	The phase of $B_m(\eta)$
	is $\Phi(\theta)=m\theta-\eta\cos\theta$.
	Its first three derivatives are
	$\Phi'(\theta)=m+\eta\sin\theta$, 
	$\Phi''(\theta)=\eta\cos\theta$, and $\Phi'''(\theta)=-\eta\sin\theta$.
	Let $\rho(u)=(1+|u|)^{-1/4}$.
	Before we prove the upper bound of $|G_d(x,t)|$, we provide two lemmas about the estimates of $B_m$ and $\rho$.
	We first estimate the upper bound of $B_m$ based on $\rho$.
	\begin{lemma}\label{lem:B bound}
		For every $m\in\Z$ and every $|\eta|\geq1$,
		\begin{equation*}\label{eq:B bound}
			|B_m(\eta)|
			\lesssim |\eta|^{-1/2}
			+|\eta|^{-1/3}\left[
			\rho\!\left(\frac{m+\eta}{|\eta|^{1/3}}\right)
			+\rho\!\left(\frac{m-\eta}{|\eta|^{1/3}}\right)
			\right].
		\end{equation*}
	\end{lemma}

	\begin{proof}
		Choose smooth functions $\psi_+$ and $\psi_-$ supported in sufficiently small neighborhoods of $\pi/2$ and $-\pi/2$, respectively, such that $\psi_+=1$ near $\pi/2$  and $\psi_-=1$ near $-\pi/2$.
		Set $\psi_0=1-\psi_+-\psi_-$.
		Then $1=\psi_++\psi_-+\psi_0$. 
		Since $\psi_0$ vanishes in neighborhoods of
		$\pm \pi/2$, we have
		$\operatorname{supp}\psi_0 \Subset
		\mathbb{T}\setminus
		\left\{-\frac{\pi}{2},\frac{\pi}{2}\right\}$.
		Since $\pm\pi/2$ are the only zeros of $\cos\theta$  on $\mathbb{T}$, there exist an open neighborhood $U$ of $\operatorname{supp}\psi_0$ and a constant $\kappa>0$ such that $ |\cos\theta|\geq \kappa$ for $\theta\in U$.
		By the compactness of $\operatorname{supp}\psi_0$, we may choose finitely many open intervals
		$J_1,\ldots,J_N\subset\mathbb{T}$ such that $\operatorname{supp}\psi_0\subset\bigcup_{\ell=1}^N J_\ell$ and $\overline{J_\ell}\subset U$ for $1\leq \ell\leq N$.
		Using the partition of unity theorem to the finite open
		cover $\mathbb{T}=
		\bigl(\mathbb{T}\setminus \operatorname{supp}\psi_0\bigr)  \cup J_1\cup\cdots\cup J_N$,  there exist nonnegative functions
		$\chi_\infty,\chi_1,\cdots,\chi_N
		\in C^\infty(\mathbb{T})$
		such that $ \operatorname{supp}\chi_\infty  \subset \mathbb{T}\setminus \operatorname{supp}\psi_0$, $\operatorname{supp}\chi_\ell\subset J_\ell$ for $1\leq\ell\leq N$, and $\chi_\infty+\sum_{\ell=1}^N\chi_\ell=1$ on $\mathbb{T}$.
		Define $ \psi_\ell\coloneqq\psi_0\chi_\ell$ for $1\leq\ell\leq N$.
		Then $\psi_\ell\in C_c^\infty(J_\ell)$.
		Moreover,  $\psi_0\chi_\infty=0$.
		Therefore, we obtain that
		\begin{align*}
			\psi_0&=\psi_0\left(\chi_\infty+\sum_{\ell=1}^N\chi_\ell\right) =\sum_{\ell=1}^N\psi_0\chi_\ell=\sum_{\ell=1}^N\psi_\ell.
		\end{align*}
		For every $1\leq\ell\leq N$ and every
		$\theta\in\operatorname{supp}\psi_\ell$, we have
		$\theta\in J_\ell\subset U$, and hence
		$ |\Phi''(\theta)|=|\eta\cos\theta| \geq\kappa|\eta|$.
		Applying Lemma~\ref{lem:van der Corput} (ii) to each local piece and summing them,
		we obtain
		\begin{equation}\label{eq:B key1}
			\int_{-\pi}^{\pi} e^{i(m\theta-\eta\cos\theta)}\psi_0(\theta)\,d\theta\lesssim|\eta|^{-1/2}.
		\end{equation}
		
		Now it suffices to deal with $\psi_+$, and the case of $\psi_-$ is similar.  The corresponding piece is $\int_{-\pi}^{\pi} e^{i(m\theta-\eta\cos\theta)}\psi_+(\theta)\,d\theta$.
		Set $\theta=\frac{\pi}{2}+y$ and $\psi(y)=\psi_+\!\left(\frac{\pi}{2}+y\right) $.
		Thus $\psi\in C_c^\infty(\R)$ and
		$\supp\psi\subset[-\delta,\delta]$
		for a fixed small $\delta>0$.  Since $\cos\left(\frac{\pi}{2}+y\right)=-\sin y$, $\Phi(\theta)=m\theta-\eta\cos\theta$ becomes
		$m\left(\frac{\pi}{2}+y\right)+\eta\sin y
		=\frac{m\pi}{2}+my+\eta\sin y$.
		The constant $m\pi/2$ only contributes to the factor $e^{im\pi/2}$, whose absolute value is $1$, so we can ignore it and consider the new phase $\phi(y)=my+\eta\sin y$. 
		Let $q=m+\eta$ and $\varepsilon=\operatorname{sgn}(\eta)\in\{\pm1\}$.
		Write $\phi(y)=qy+\varepsilon|\eta|(\sin y-y)$.
		We now prove that
		\begin{equation}\label{eq:B key2}
			\left|\int_{-\delta}^{\delta} e^{i\phi(y)}\psi(y)\,d y\right|
			\lesssim |\eta|^{-1/3}\rho\!\left(\frac{q}{|\eta|^{1/3}}\right).
		\end{equation}
		For a fixed constant $C_0>0$, we divide the proof into two cases.
		
		{\flushleft \bf Case A.} $|q|\leq C_0|\eta|^{1/3}$.
		
		In this case,
		\begin{equation}\label{eq:rho q}
			\rho\!\left(\frac{q}{|\eta|^{1/3}}\right)
			=\left(1+\left|\frac{q}{|\eta|^{1/3}}\right|\right)^{-1/4}
			\geq(1+C_0)^{-1/4}.
		\end{equation}
		Since $\phi'''(y)=-\varepsilon|\eta|\cos y$ and $\supp\psi\subset[-\delta,\delta]$, we have  $|\phi'''(y)|\gtrsim|\eta|$.
		Therefore, by Lemma~\ref{lem:van der Corput} (ii), we obtain
		$	\left|\int_{-\pi}^{\pi} e^{i\phi(y)}\psi(y)\,d y\right|\lesssim |\eta|^{-1/3}$,
		which, together with~\eqref{eq:rho q}, implies~\eqref{eq:B key2}.
		
		{\flushleft \bf Case B.} $|q|>C_0|\eta|^{1/3}$.
		
		In this case,
		\begin{align}\label{eq:1/q}
			|q|^{-1}&\lesssim(|\eta||q|)^{-1/4}
			=|\eta|^{-1/3}\left(\frac{|q|}{|\eta|^{1/3}}\right)^{-1/4}\notag\\
			&= |\eta|^{-1/3}\left(\frac{|\eta|^{1/3}}{|q|}+1\right)^{\frac{1}{4}}\left(\frac{|q|}{|\eta|^{1/3}}+1\right)^{-\frac{1}{4}}\\
			&<\left(\frac{C_0+1}{C_0}\right)^\frac{1}{4}|\eta|^{-1/3} \rho\!\left(\frac{q}{|\eta|^{1/3}}\right)\notag.
		\end{align}
		Let $h(y)=1-\cos y$.
		Then $\phi'(y)=q-\varepsilon|\eta| h(y)$.
		On $[0,2\delta]$, $h$ is strictly increasing and there exist constants $c,C>0$ such that
		\begin{equation}\label{eq:h bound}
			cy^2\leq h(y)\leq Cy^2,\quad c|y|\leq|\sin y|\leq C|y|.
		\end{equation}
		If $\varepsilon q<0$, then $q$ and $-\varepsilon|\eta| h(y)$ have the same sign, and hence $|\phi'(y)|\geq|q|$. Splitting at $y=0$ makes $\phi'$ monotone on each half interval, and using Lemma~\ref{lem:van der Corput} (i), we have 
		\begin{equation*}
			\left|\int_{-\delta}^{\delta} e^{i\phi(y)}\psi(y)\,d y\right|\lesssim |q|^{-1}.
		\end{equation*}
		Thus \eqref{eq:B key2} follows from \eqref{eq:1/q}.
		
		It remains to consider the case $\varepsilon q>0$.
		
		{\flushleft \bf Subcase B.1.} $|q|/|\eta|<h(2\delta)$.
		
		Let $y_0\in(0,2\delta)$ be the unique solution of $h(y_0)=|q|/|\eta|$.
		By~\eqref{eq:h bound}, we have
		\begin{equation}\label{eq:y0}
			y_0\asymp\left(\frac{|q|}{|\eta|}\right)^{1/2}.
		\end{equation}
		We claim that there exist constants $0<c_-<1<c_+$, which only depend  on $\delta$, so that
		\begin{equation}\label{eq:detla h}
			|h(y)-h(y_0)|\geq\frac{1}{2} h(y_0)
		\end{equation}
		whenever $|y|\leq2\delta$ and $|y|\notin[c_-y_0,c_+y_0]$.
		Choose $c_-$ so small that $Cc_-^2\leq c/2$.
		If $|y|\leq c_-y_0$, then
		$h(y)\leq Cy^2\leq Cc_-^2y_0^2\leq \frac{c}{2}y_0^2\leq\frac{1}{2}h(y_0)$,
		and hence $|h(y)-h(y_0)|\geq\frac{1}{2}h(y_0)$.
		Next, choose $c_+$ so large that $cc_+^2\geq2C$.  
		If $|y|\geq c_+y_0$, then
		$h(y)\geq cy^2\geq cc_+^2y_0^2
		\geq2Cy_0^2\geq2h(y_0)$, which gives 
		$|h(y)-h(y_0)|\geq h(y_0)$.
		Thus \eqref{eq:detla h} holds.
		
		Let $I_1=\{-\delta\leq y\leq\delta: |y|\notin[c_-y_0,c_+y_0]\}$. 
		Assume first $y\in I_1$. Since $\varepsilon q>0$, we have $q=\varepsilon|q|$. 
		As $h(y_0)=|q|/|\eta|$, 
		$$\phi'(y)=q-\varepsilon|\eta| h(y)
		=\varepsilon|\eta|\bigl(h(y_0)-h(y)\bigr).$$
		Hence, by \eqref{eq:detla h},
		$$
		|\phi'(y)|
		=|\eta||h(y)-h(y_0)|
		\geq\frac{1}{2}|\eta| h(y_0)
		=\frac{1}{2}|q|.
		$$
		After splitting at $0$ and the endpoints $\pm c_-y_0,\pm c_+y_0$, each remaining interval lies entirely in $y>0$ or in $y<0$.  
		Since $\phi''(y)=-\varepsilon|\eta|\sin y$ has a fixed sign on each such half interval, $\phi'$ is monotone on each interval.  
		By Lemma~\ref{lem:van der Corput} (i), 
		\begin{equation}\label{eq:I_1}
			\left|\int_{I_1} e^{i\phi(y)}\psi(y)\,d y\right|\lesssim |q|^{-1}.
		\end{equation}

		Now assume $y\in I_2=[-\delta,\delta]\setminus I_1$. 
		Then $|y|\geq c_-y_0$.  
		Since $|y|\leq2\delta$ and $\delta$ is small, \eqref{eq:h bound} gives $|\sin y|\geq c|y|$.  
		Thus,
		$$
		|\phi''(y)|=|\eta||\sin y|\geq c|\eta||y| \geq c c_-|\eta| y_0\gtrsim|\eta| y_0.
		$$
		Using~\eqref{eq:y0}, we have 
		$|\eta| y_0\asymp |\eta|\left(\frac{|q|}{|\eta|}\right)^{1/2}=(|\eta||q|)^{1/2}$.
		Therefore, $|\phi''(y)|\gtrsim(|\eta||q|)^{1/2}$.
		By Lemma~\ref{lem:van der Corput} (ii), we obtain
		\begin{equation}\label{eq:I_2}
			\left|\int_{I_2}e^{i\phi(y)}\psi(y)\,d y\right|\lesssim (|\eta||q|)^{-1/4}.
		\end{equation}
		Sum~\eqref{eq:I_1} and~\eqref{eq:I_2}, and then~\eqref{eq:B key2} follows from~\eqref{eq:1/q}.

		{\flushleft \bf Subcase B.2.} $|q|/|\eta|\geq h(2\delta)$.
		
		Recall that $\psi$ is supported in $|y|\leq\delta$.
		We should give a lower bound for $|\phi'(y)|$ when $|y|\leq\delta$.
		Since $\varepsilon q>0$ and $\phi'(y)=q-\varepsilon|\eta| h(y)$, we have
		\begin{equation*}
			|\phi'(y)|=|q|-|\eta|h(y)=|q|\left((1-\frac{|\eta|}{|q|}h(y)\right).
		\end{equation*}
		Combined with $|q|/|\eta|\geq h(2\delta)$ and $h(y)\leq h(\delta)$ for $|y|\leq\delta$, it gives 
		\begin{equation*}
			|\phi'(y)|\geq |q|\left(1-\frac{h(\delta)}{h(2\delta)}\right).
		\end{equation*}
		Note that $h(\delta)<h(2\delta)$ and $\phi''(y)=-\varepsilon|\eta|\sin y$ has a fixed sign on each half interval after splitting at $y=0$. 
		Thus $|\phi'|\gtrsim |q|$ and is monotone on each piece. 
		Applying Lemma~\ref{lem:van der Corput} (i), we have
		\begin{equation*}
			\left|\int_{-\pi}^{\pi} e^{i\phi(y)}\psi(y)\,d y\right|\lesssim |q|^{-1}.
		\end{equation*}
		Once again, \eqref{eq:B key2} follows from \eqref{eq:1/q}.
		
		Finally, at $-\pi/2$ the corresponding parameter is $m-\eta$, and the proof is similar. 
		Combining it with~\eqref{eq:B key1} and~\eqref{eq:B key2}, we complete the proof.
	\end{proof}

	Then we work on the function
	$\rho(u)=(1+|u|)^{-1/4}$.
	
	\begin{lemma}\label{lem:prod rho}
		Let $J\subset\R$ be an interval of length at most $L$, with $L\geq1$.  For $k\geq1$, arbitrary $a_1,\dots,a_k\in\R$, and arbitrary signs $\sigma_i\in\{\pm1\}$,
		\begin{equation*}\label{eq:prod rho}
			\int_J\prod_{i=1}^k\rho(\sigma_iu+a_i)\,d u\lesssim_k
			\begin{cases}
				L^{1-k/4},&1\leq k<4,\\
				\log(2+L),&k=4,\\
				1,&k>4.
			\end{cases}
		\end{equation*}
	\end{lemma}
	
	\begin{proof}
		By H\"older's inequality, we have
		\begin{equation}\label{eq:Holder}
			\int_J\prod_{i=1}^k\rho(\sigma_iu+a_i)\,d u
			\leq\prod_{i=1}^k\left(\int_J\rho(\sigma_iu+a_i)^k\,d u\right)^{1/k}.
		\end{equation}
		Fix $a\in\R$ and $\sigma\in\{\pm1\}$.  Under the change of variables $v=\sigma u+a$,
		the image of $J$ is another interval $J'$ with exactly the same length as $J$.  
		Therefore
		$$
		\int_J\rho(\sigma u+a)^k\,d u
		=\int_{J'}\rho(v)^k\,d v.
		$$
		It remains to bound the largest possible integral of $\rho(v)^k=(1+|v|)^{-k/4}$
		over an interval of length at most $L$.  
		This function is largest near $0$ and decreases as $|v|$ increases.  Hence
		$$
		\sup_{|J|\leq L}\int_J(1+|v|)^{-k/4}\,d v
		\leq\int_{-\frac{L}{2}}^{\frac{L}{2}}(1+|v|)^{-k/4}\,d v.
		$$
		The last integral is $O_k(L^{1-k/4})$ for $k<4$, $O(\log(2+L))$ for $k=4$, and $O_k(1)$ for $k>4$.  
		Substituting into~\eqref{eq:Holder} yields the desired bound.
	\end{proof}
	
	Now we give the proof of the upper bound.
	\begin{proof}[Proof of the upper bound in Theorem~\ref{thm:main}]
		Let $T_0=\max\{1,c_1^{-1}\}$ and $E_t=\{\eta\in\R:c_1|t|\leq|\eta|\leq c_2|t|\}$,
		where $c_2>c_1>0$ are the constants provided by Lemma~\ref{lem:hat F}.
		
		If $|t|\leq T_0$, then we use the elementary inequality $\left|\frac{\sin(t\omega)}{\omega}\right|\leq|t|$,
		which gives $|G_d(x,t)|\leq |t|\leq T_0$.
		Since $(1+|t|)^{-\sigma_d}\geq (1+T_0)^{-\sigma_d}$ for $t\in [-T_0,T_0]$, we obtain~\eqref{eq:uppermain} by choosing a constant $C_d$ sufficiently large; for instance, one may take $C_d=T_0(1+T_0)^{\sigma_d}$.
		
		By symmetry, it remains to prove~\eqref{eq:uppermain} for $t\geq T_0$. Moreover, since 
		$t\leq1+|t|\leq (1+\frac{1}{T_0})t$ for $t\geq T_0$, we can replace \eqref{eq:uppermain} by 
		\begin{equation}\label{eq:uppermain1}
			|G_d(x,t)|\le C_dt^{-\sigma_d},
			\quad x\in\Z^d.
		\end{equation}   
		If $|x|\geq c_dt$, then Lemma~\ref{lem:schultz} (i) gives
		$|G_d(x,t)|\lesssim_d t^{-d/2}$,
		which is stronger than~\eqref{eq:uppermain1} because $\frac{d}{2}-\sigma_d=\frac{d-1}{6}>0$. Assume that $|x|\leq c_dt$.  Lemma \ref{lem:schultz} (ii) gives $O_d(t^{-(d-1)})$ for $|K_d^{\mathrm{low}}|$.  
		It suffices to prove  $|K_d^{\mathrm{hi}}(x,t)|\lesssim_d t^{-\sigma_d}$ for $x\in\Z^d$
		for $d\geq 5$.
		
		By Lemma~\ref{lem:factorization} and Lemma~\ref{lem:hat F}, up to an $O_N(t^{-N})$ error for arbitrary $N$, we have
		\begin{equation}\label{eq:K_d^hi bound}
			|K_d^{\mathrm{hi}}(x,t)|
			\lesssim_d t^{-1/2}\int_{E_t}\prod_{i=1}^d|B_{x_i}(\eta)|\,d\eta.
		\end{equation}
		Since $|\eta|\geq c_1t\geq1$ for every $\eta\in E_t$, we can apply Lemma~\ref{lem:B bound} to every $|B_{x_i}(\eta)|$ and multiply them, yielding
		\begin{equation}\label{eq:prod B_x_i}
			\prod_{i=1}^d|B_{x_i}(\eta)|\lesssim \prod_{i=1}^d\left[|\eta|^{-1/2}+|\eta|^{-1/3}\rho\!
			\left(\frac{x_i+\eta}{|\eta|^{1/3}}\right)
			+|\eta|^{-1/3}\rho\!\left(\frac{x_i-\eta}{|\eta|^{1/3}}\right)\right].
		\end{equation} 
		For each coordinate $x_i$, we can choose either its regular term $|\eta|^{-1/2}$ or one of its two $\rho$-terms $|\eta|^{-1/3}\rho\!\left(\frac{x_i+\sigma\eta}{|\eta|^{1/3}}\right)$, where $\sigma\in \{\pm1\}$.
		Choose a term of \eqref{eq:prod B_x_i} that contains exactly $k$ $\rho$-factors and $d-k$ regular factors.  Since $|\eta|\asymp t$ on $E_t$, there is a constant $C$ such that
		\begin{equation*}
			\rho\!\left(\frac{x+\sigma\eta}{|\eta|^{1/3}}\right)
			\leq C\rho\!\left(\frac{x+\sigma\eta}{t^{1/3}}\right),
			\qquad \eta\in E_t.
		\end{equation*}
		Hence the chosen term together with the coefficient $t^{-1/2}$ in \eqref{eq:K_d^hi bound} is bounded by
		\begin{equation}\label{eq:k term1}
			Ct^{-1/2}t^{-(d-k)/2}t^{-k/3}
			\int_{E_t}\prod_{\ell=1}^k
			\rho\!\left(\frac{x_{i_\ell}+\sigma_\ell\eta}{t^{1/3}}\right)\,d\eta.
		\end{equation}
		If $k=0$, all factors are regular. 
		Then integral in \eqref{eq:k term1} is simply $\int_{E_t}d\eta\lesssim t$, and hence \eqref{eq:k term1} is $O(t^{-(d-1)/2})$.
		
		Suppose $k\geq1$. 
		The set $E_t$ is contained in an interval of length $L_E\lesssim t$. 
		Under the change of variables $\eta=t^{1/3}u$, \eqref{eq:k term1} becomes
		\begin{equation}\label{eq:k term2}
			Ct^{-d/2+k/6-1/6}
			\int_{D_t}\prod_{\ell=1}^k\rho(\sigma_\ell u+a_\ell)\,d u,
			\quad a_\ell=\frac{x_{i_\ell}}{t^{1/3}},
		\end{equation}
		where $D_t$ is contained in an interval of length $L_D\lesssim t^{2/3}$.
		If $1\leq k<4$, Lemma~\ref{lem:prod rho} gives an additional factor $t^{2/3-k/6}$, so \eqref{eq:k term2} is $O(t^{-(d-1)/2})$.  
		If $k=4$, it is $O\bigl(t^{-(d-1)/2}\log(2+t)\bigr)$.
		If $k\geq5$, the integral in \eqref{eq:k term2} is uniformly bounded and then \eqref{eq:k term2} is $O_d\bigl(t^{-(3d-k+1)/6}\bigr)$.
		Since $k\leq d$, $\frac{3d-k+1}{6}\geq\frac{2d+1}{6}=\sigma_d$.
		Finally, for $d\geq5$, $\frac{d-1}{2}-\sigma_d=\frac{d-4}{6}>0$,
		so
		\begin{equation}\label{eq:log of d=4}
			t^{-(d-1)/2}\log(2+t)\lesssim_d t^{-\sigma_d}.
		\end{equation}
		Summing these bounds for the finitely many terms in \eqref{eq:K_d^hi bound} yields $|K_d^{\mathrm{hi}}(x,t)|\lesssim_d t^{-\sigma_d}$ for $x\in\Z^d$ and $d\geq 5$.
		This completes the proof.
	\end{proof}
	
	Taking $d=4$ in the above proof, we can obtain the following corollary immediately (We only need to note that the inequality in \eqref{eq:log of d=4} reverses direction). 
	\begin{corollary}\label{corollary}
		$|G_4(x,t)|\lesssim (1+|t|)^{-3/2}\log(2+|t|)$ for $(x,t)\in \Z^4\times \R$.
	\end{corollary}
	
	\begin{remark}\label{remark}
		The logarithm for $d=4$ comes from the integral of the function $\rho^k(u)=(1+|u|)^{-k/4}$, and when $d\geq 5$, it is dominated by lower order term; see \eqref{eq:log of d=4}.
	\end{remark}

	\section{Sharpness of the estimates}\label{sec:6}
	Set $\xi^*=(\pi/2,\ldots,\pi/2)$, $w_*=\omega(\xi^*)=\sqrt{2d}$, and $v^*=w_*^{-1}(1,\dots,1)$.
	Let $e$ be the unit vector $d^{-1/2}(1,\dots,1)^\top$. 
	The orthogonal projection matrix onto the line spanned by $e$ and onto its orthogonal complement $e^\perp=\{z\in\R^d:\sum_{j=1}^dz_j=0\}$ are respectively
	$P_\parallel=ee^\top$ and $P_\perp=I-ee^\top$.
	For $t\geq1$, define the operator
	$A_t\coloneqq t^{1/3}P_\perp+t^{1/2}P_\parallel$.
	Then $A_t$ is positive definite and self-adjoint, with eigenvalues $t^{1/3}$ on $e^\perp$ and $t^{\frac{1}{2}}$ on $\operatorname{span}\{e\}$.
	Its determinant is
	\begin{equation}\label{eq:detAt}
		\det A_t=t^{(d-1)/3+1/2}=t^{\sigma_d}.
	\end{equation}
	
	Every $\zeta\in\R^d$ has a unique decomposition $\zeta=z+re$ with $z\in e^\perp$ and $r=e^\top\zeta$. Define the polynomial
	\begin{equation*}
		\Phi_\infty(\zeta)=\frac{d}{2w_*^3}r^2+\frac{1}{6w_*}\sum_{j=1}^d z_j^3.
	\end{equation*}
	We need two small neighborhood conditions.  First, since $e^\top\bigl(v^*+\nabla\omega(\xi^*+\theta)\bigr)=\sqrt{2}$ at $\theta=0$,
	there is a radius $\delta_1>0$ such that
	\begin{equation}\label{eq:detla1}
		e^\top\bigl(v^*+\nabla\omega(\xi^*+\theta)\bigr)\geq 1 \quad \text{for } |\theta|\leq\delta_1.
	\end{equation}  
	Second, since $\Phi_\infty(0)=0$, there is a radius $\delta_2>0$ such that
	\begin{equation}\label{eq:detla2}
		|\Phi_\infty(\zeta)|\leq\pi/4 \quad \text{for } |\zeta|\leq\delta_2.
	\end{equation}
	Fix a small radius $0<\delta<\min\{\delta_1,\delta_2,1\}$.
	Now we can choose a nonnegative bump function $\widehat W\in C_c^\infty(\R^d)\subset\Sch(\R^d)$ such that $\widehat W\not\equiv0$, $\widehat W(-\zeta)=\widehat W(\zeta)$, and $\supp\widehat W\subset B(0,\delta)$.
	Let $W$ be its inverse Fourier transform.
	Since the Fourier transform is a bijective map of $\Sch(\R^d)$, $W$ is also in $\Sch(\R^d)$. Define the lattice average
	\begin{equation}\label{eq:St}
		S_t=\sum_{x\in\Z^d}W\bigl(A_t^{-1}(x-tv^*)\bigr)e^{-ix\cdot\xi^*}G_d(x,t).
	\end{equation}
	First, we give two lemmas about the lattice average $S_t$.
	\begin{lemma}\label{lem:total W}
		There exists a constant $C_{\mathrm{sum}}>0$ such that, for all $t\geq1$ and $y\in\R^d$,
		\begin{equation*}
			\sum_{x\in\Z^d}|W(A_t^{-1}(x-y))|\leq C_{\mathrm{sum}}\det A_t.
		\end{equation*}
		Moreover,
		\begin{equation}\label{eq:St bound}
			|S_t|\leq C_{\mathrm{sum}}t^{\sigma_d}\sup_{x\in\Z^d}|G_d(x,t)|.
		\end{equation}
	\end{lemma}
	
	\begin{proof}
		Note that all the eigenvalues of $A_t$ are at least one.  
		Fix $d$ and $N>d$.  
		Since $W\in\Sch(\R^d)$, by the definition, it gives $|W(z)|\lesssim (1+|z|)^{-N}$.
		For $x\in\Z^d$, let $Q_x=x+[-1/2,1/2]^d$.  If $u\in Q_x$, then
		$$
		|A_t^{-1}(u-y)|\leq|A_t^{-1}(x-y)|+\frac{\sqrt{d}}{2}\|A_t^{-1}\|
		\leq|A_t^{-1}(x-y)|+\frac{\sqrt{d}}{2}.
		$$
		Therefore, $|W(A_t^{-1}(x-y))|\lesssim(1+|A_t^{-1}(x-y)|)^{-N}\lesssim(1+|A_t^{-1}(u-y)|)^{-N}$.
		Since every cube $Q_x$ has volume $1$, it  follows that
		$$
		|W(A_t^{-1}(x-y))|=\int_{Q_x}|W(A_t^{-1}(x-y))|du
		\lesssim\int_{Q_x}(1+|A_t^{-1}(u-y)|)^{-N}du.
		$$
		Summing in $x$, we have
		$$
		\sum_{x\in\Z^d}|W(A_t^{-1}(x-y))|\lesssim\int_{\R^d}(1+|A_t^{-1}(u-y)|)^{-N}du.
		$$
		Under the change of variables $z=A_t^{-1}(u-y)$ and $N>d$, we obtain
		$$
		\sum_{x\in\Z^d}|W(A_t^{-1}(x-y))|\lesssim\det A_t\int_{\R^d}(1+|z|)^{-N}\,d z \lesssim \det A_t.
		$$
		The bound \eqref{eq:St bound} follows directly from \eqref{eq:detAt} and \eqref{eq:St}.
	\end{proof}
	
	\begin{lemma}\label{lem:poisson}
		For all sufficiently large $t$,
		\begin{equation}\label{eq:St integral}
			S_t=\frac{1}{(2\pi)^d}\int_{\R^d}
			\widehat W(-\zeta)\exp\bigl(itv^*\cdot A_t^{-1}\zeta\bigr)
			\frac{\sin\bigl(t\omega(\xi^*+A_t^{-1}\zeta)\bigr)}
			{\omega(\xi^*+A_t^{-1}\zeta)}\,d\zeta.
		\end{equation}
	\end{lemma}
	
	\begin{proof}
		For $y,\theta\in\mathbb R^d$, define
		$f_{y,\theta}(u)=W\!\left(A_t^{-1}(u-y)\right)e^{iu\cdot\theta}$.
		Since $W\in\mathcal S(\mathbb R^d)$, we have $f_{y,\theta}\in\mathcal S(\mathbb R^d)$. Then Fourier transform gives
		\begin{equation*}
			\widehat f_{y,\theta}(\eta)=\int_{\mathbb R^d}
			W\!\left(A_t^{-1}(u-y)\right)e^{iu\cdot(\theta-\eta)}\,du.
		\end{equation*}
		Since $A_t$ is positive definite and self-adjoint, under the change of variables $z=A_t^{-1}(u-y)$,
		we have $du=(\det A_t)\,dz$ and
		$$
		iu\cdot(\theta-\eta)=iy\cdot(\theta-\eta)+i(A_tz)\cdot(\theta-\eta)=iy\cdot(\theta-\eta)-iz\cdot A_t(\eta-\theta),
		$$
		where $\cdot$ denotes the inner product of two vectors.
		Therefore,
		\begin{align*}
			\widehat f_{y,\theta}(\eta)
			&=(\det A_t)e^{iy\cdot(\theta-\eta)}\int_{\mathbb R^d}
			W(z)e^{-iz\cdot A_t(\eta-\theta)}\,dz \\
			&=(\det A_t)e^{iy\cdot(\theta-\eta)}
			\widehat W\!\left(A_t(\eta-\theta)\right).
		\end{align*}
		Hence  Lemma~\ref{lem:Poisson formula} gives
		\begin{equation}\label{eq:Poisson W}
			\sum_{x\in\Z^d}
			W\!\left(A_t^{-1}(x-y)\right)e^{ix\cdot\theta}
			=\det A_t\sum_{k\in\Z^d}e^{iy\cdot(\theta-2\pi k)}
			\widehat W\!\left(A_t(2\pi k-\theta)\right).
		\end{equation}
		
		We now transform $S_t$ into a form to which~\eqref{eq:Poisson W} can be applied. 
		Recall that
		$$
		S_t=\sum_{x\in\Z^d}W\!\left(A_t^{-1}(x-tv^*)\right)
		e^{-ix\cdot\xi^*}G_d(x,t),
		$$
		where
		$$
		G_d(x,t)=\frac{1}{(2\pi)^d}\int_{\T^d}e^{ix\cdot\xi}
		\frac{\sin(t\omega(\xi))}{\omega(\xi)}\,d\xi.
		$$
		Substituting the latter identity gives
		\begin{equation*}
			S_t=\frac{1}{(2\pi)^d}\sum_{x\in\Z^d}\int_{\T^d}
			W\!\left(A_t^{-1}(x-tv^*)\right)e^{ix\cdot(\xi-\xi^*)}
			\frac{\sin(t\omega(\xi))}{\omega(\xi)}\,d\xi.
		\end{equation*}
		Since $\left|\frac{\sin(t\omega(\xi))}{\omega(\xi)}\right|\leq |t|$ 
		and Lemma~\ref{lem:total W} gives
		$\sum\limits_{x\in\Z^d} \left|W\!\left(A_t^{-1}(x-tv^*)\right)\right|<+\infty$ for fixed $t$,
		we have 
		\begin{equation*}
			\sum_{x\in\Z^d}\int_{\T^d}\left|
			W\!\left(A_t^{-1}(x-tv^*)\right)e^{ix\cdot(\xi-\xi^*)}
			\frac{\sin(t\omega(\xi))}{\omega(\xi)}\right|\,d\xi<+\infty.
		\end{equation*}
		Hence, by Fubini's theorem, we may interchange the lattice sum 
		and the torus integral:
		\begin{equation}\label{eq:St Fubini}
			S_t=\frac{1}{(2\pi)^d}\int_{\T^d}\left[\sum_{x\in\Z^d}
			W\!\left(A_t^{-1}(x-tv^*)\right)e^{ix\cdot(\xi-\xi^*)}\right]
			\frac{\sin(t\omega(\xi))}{\omega(\xi)}\,d\xi.
		\end{equation}
		Notice that the function appearing in the integral over $\T^d$ is $2\pi$-periodic in each coordinate of $\xi$.  Then we can represent $\T^d$ by the domain $\xi^*+[-\pi,\pi]^d$. Set $\theta=\xi-\xi^*$, then $\theta\in[-\pi,\pi]^d$. Applying \eqref{eq:Poisson W} with $y=tv^*$, 
		we obtain
		\begin{equation*}
			\sum_{x\in\Z^d}W\!\left(A_t^{-1}(x-tv^*)\right)e^{ix\cdot(\xi-\xi^*)}
			=\det A_t\sum_{k\in\Z^d}e^{itv^*\cdot(\theta-2\pi k)}\widehat W\!\left(A_t(2\pi k-\theta)\right).
		\end{equation*}
		
		We claim that, for all sufficiently large $t$, every term with $k\neq0$ vanishes. 
		In fact, if $k\neq0$, then there exists some $j$ such that $k_j\neq0$. 
		Since $|\theta_j|\leq\pi$, $|2\pi k_j-\theta_j|\geq2\pi|k_j|-|\theta_j|\geq \pi$.
		Hence $|2\pi k-\theta|\geq\pi$.
		The smallest eigenvalue of $A_t$ is $t^{1/3}$, and it follows that 
		$|A_t(2\pi k-\theta)|\geq t^{1/3}|2\pi k-\theta|\geq \pi t^{1/3}$.
		Since $\operatorname{supp}\widehat W\subset B(0,\delta)$, we have $\widehat W\!\left(A_t(2\pi k-\theta)\right)=0$ for  $k\neq0$ and all sufficiently large $t$.
		Thus,
		$$
		\sum_{x\in\Z^d}W\!\left(A_t^{-1}(x-tv^*)\right)e^{ix\cdot(\xi-\xi^*)}
		=(\det A_t)e^{itv^*\cdot(\xi-\xi^*)}\widehat W\!\left(-A_t(\xi-\xi^*)\right).
		$$
		Substituting this identity into \eqref{eq:St Fubini}, we obtain
		\begin{equation*}
			S_t=\frac{\det A_t}{(2\pi)^d}\int_{\xi^*+[-\pi,\pi]^d}
			e^{itv^*\cdot(\xi-\xi^*)}\widehat W\!\left(-A_t(\xi-\xi^*)\right)
			\frac{\sin(t\omega(\xi))}{\omega(\xi)}\,d\xi.
		\end{equation*}
		Under the change of variables $\zeta=A_t(\xi-\xi^*)$, then 
		$\xi=\xi^*+A_t^{-1}\zeta$ and $d\xi=(\det A_t)^{-1}\,d\zeta$.
		Therefore, we get
		\begin{equation*}
			S_t=\frac{1}{(2\pi)^d}\int_{A_t[-\pi,\pi]^d}
			\widehat W(-\zeta)e^{itv^*\cdot A_t^{-1}\zeta}
			\frac{\sin\!\left(t\omega(\xi^*+A_t^{-1}\zeta)\right)}{\omega(\xi^*+A_t^{-1}\zeta)}\,d\zeta.
		\end{equation*}
		Since $\operatorname{supp}\widehat W\subset B(0,\delta)$ and the largest eigenvalue of $A_t^{-1}$ is $t^{-1/3}$,
		we have $|A_t^{-1}\zeta|\leq\delta t^{-1/3}$ for $\zeta\in\operatorname{supp}\widehat W$.
		Thus, for all sufficiently large $t$, $A_t^{-1}\zeta\in[-\pi,\pi]^d$, 
		and hence $\operatorname{supp}\widehat W\subset A_t[-\pi,\pi]^d$.
		We may extend the domain of integration to $\mathbb R^d$, obtaining
		\begin{equation*}
			S_t=\frac{1}{(2\pi)^d}\int_{\R^d}
			\widehat W(-\zeta)e^{itv^*\cdot A_t^{-1}\zeta}
			\frac{\sin\!\left(t\omega(\xi^*+A_t^{-1}\zeta)\right)}
			{\omega(\xi^*+A_t^{-1}\zeta)}\,d\zeta,
		\end{equation*}
		as desired.
	\end{proof}
	
	Next, we inspect the phase near $\xi^*$.  
	Define
	\begin{equation}\label{eq:phi*}
		\phi_*(\theta)=v^*\cdot\theta-\bigl(\omega(\xi^*+\theta)-w_*\bigr).
	\end{equation}
	Then $\phi_*(0)=0$ and $\nabla\phi_*(0)=0$.
	
	\begin{lemma}\label{lem:phi* Tayler}
		For $\theta$ sufficiently small,
		\begin{equation}\label{eq:phi* Tayler}
			\phi_*(\theta)
			=\frac{1}{2w_*^3}\left(\sum_{j=1}^d\theta_j\right)^2
			+\frac{1}{6w_*}\sum_{j=1}^d\theta_j^3
			-\frac{1}{2w_*^5}\left(\sum_{j=1}^d\theta_j\right)^3
			+O(|\theta|^4).
		\end{equation}
	\end{lemma}
	
	\begin{proof}
		Since $\omega(\xi^*+\theta)^2=2d+2\sum_{j=1}^d\sin\theta_j$, Taylor's formula gives
		$$
		2\sum_j\sin\theta_j=2\sum_j\theta_j-\frac{1}{3}\sum_j\theta_j^3+O(|\theta|^5).
		$$
		Let $q=2\sum_{j=1}^d\sin\theta_j$.
		Using $(w_*^2+q)^{1/2}=w_*+\frac{q}{2w_*}-\frac{q^2}{8w_*^3}+\frac{q^3}{16w_*^5}+O(q^4)$
		and keeping terms up to third order, we obtain
		$$
		\omega(\xi^*+\theta)-w_*
		=\frac{1}{w_*}\sum_j\theta_j-\frac{1}{6w_*}\sum_j\theta_j^3-\frac{1}{2w_*^3}\left(\sum_j\theta_j\right)^2+\frac{1}{2w_*^5}\left(\sum_j\theta_j\right)^3+O(|\theta|^4).
		$$
		Combining it with \eqref{eq:phi*} and $v^*\cdot\theta=\frac{1}{w_*}\sum_j\theta_j$, we obtain~\eqref{eq:phi* Tayler} immediately.
	\end{proof}
	
	Then we give the asymptotic of the phase.
	\begin{lemma}\label{lem:limitphase}
		For every $\zeta\in\R^d$, as $t\to+\infty$,  $t\phi_*(A_t^{-1}\zeta)\longrightarrow\Phi_\infty(\zeta)$ and $\omega(\xi^*+A_t^{-1}\zeta)^{-1}\longrightarrow w_*^{-1}$. Moreover, both convergences are uniform on compact subsets of $\R^d$.
	\end{lemma}
	
	\begin{proof}
		Every $\zeta\in\R^d$ decomposes uniquely as $\zeta=z+re$, $z=P_\perp\zeta\in e^\perp$, and $r=e\cdot\zeta$.
		By the definition of $A_t$,
		$	A_t^{-1}\zeta=t^{-1/3}z+t^{-1/2}re$.
		Since $z\in e^\perp$, $\sum_jz_j=0$.  
		It follows that $\sum_j(A_t^{-1}\zeta)_j=\sqrt{d}\,t^{-1/2}r$.
		Now we apply Lemma~\ref{lem:phi* Tayler} with $\theta=A_t^{-1}\zeta$  for sufficiently large $t$.
		So the first term in~\eqref{eq:phi* Tayler}  times $t$ equals $dr^2/(2w_*^3)$.  
		Also, the second term in \eqref{eq:phi* Tayler} times $t$ is
		$$
		\frac{t}{6w_*}\sum_{j=1}^d(t^{-1/3}z_j+t^{-1/2}r/\sqrt{d})^3
		=\frac{1}{6w_*}\sum_{j=1}^dz_j^3+O(t^{-1/6}).
		$$
		The term involving $(\sum_j\theta_j)^3$ is $O(t^{-1/2})$, and $tO(|\theta|^4)=O(t^{-1/3})$.  
		This proves $t\phi_*(A_t^{-1}\zeta)\\\to\Phi_\infty(\zeta)$.  
		The second one follows from continuity of $\omega$ and the fact that $A_t^{-1}\zeta\to0$. 
        Moreover, if $\zeta$ ranges over a fixed compact set, then $z$ and $r$ are uniformly bounded. 
        So both convergences are uniform.
	\end{proof}
	
	Using $\sin x=\frac{e^{ix}-e^{-ix}}{2i}$ in~\eqref{eq:St integral}, we have
	$S_t=\frac{1}{2i(2\pi)^d}(I_t^+-I_t^-)$,
	where, with $\theta_t(\zeta)=A_t^{-1}\zeta$,
	\begin{equation*}\label{eq:Ipm}
		I_t^\pm=\int_{\R^d}\widehat W(-\zeta)
		\frac{\exp\left(it[v^*\cdot\theta_t(\zeta)\pm\omega(\xi^*+\theta_t(\zeta))]\right)}
		{\omega(\xi^*+\theta_t(\zeta))}\,d\zeta.
	\end{equation*}
	Instead of estimating $S_t$ directly, we consider $I_t^\pm$ as follows.
	\begin{lemma}\label{lem:I}
		As $t\to+\infty$, $I_t^+=O(t^{-1/2})$ and $e^{itw_*}I_t^-\to C_W$
		for some constant $C_W\neq 0$.
	\end{lemma}
	
	\begin{proof}
		Write $\zeta=z+re$ and set $\Psi_t(\zeta)=t\bigl[v^*\cdot\theta_t(\zeta)+\omega(\xi^*+\theta_t(\zeta))\bigr]$.
		Since $\theta_t(\zeta)=A_t^{-1}\zeta$, we have $\partial_r\theta_t=t^{-1/2}e$. 
		Then $\partial_r\Psi_t=t^{1/2}e^\top\bigl(v^*+\nabla\omega(\xi^*+\theta_t)\bigr)$.
		Because $\zeta\in\supp\widehat W\subset B(0,\delta)$ and the eigenvalues of $A_t$ are at least $1$, we have $|\theta_t|\leq|\zeta|<\delta$.  
		Therefore, our choice of $\delta$ and~\eqref{eq:detla1} gives $|\partial_r\Psi_t|\geq t^{1/2}$.
		For the oscillatory integral $I_t^+$, $\widehat W(-\zeta)\,\omega(\xi^*+\theta_t)^{-1}$ is supported in a fixed compact set, and all of its $r$-derivatives
		are uniformly bounded.  
		Since $\partial_r\theta_t=t^{-1/2}e$, we get
		$|\partial_r^j\Psi_t|\lesssim_j t^{1-j/2}\lesssim_j 1$ for $j\geq2$.
		Thus, for each fixed $z$, we can apply Lemma~\ref{lem:BKY} on $I_t^+$ in the $r$-variable with $Y=1$ and $R=t^{1/2}$.  Taking $A=1$ gives a uniform $O(t^{-1/2})$ bound.
		Finally, $z$ ranges over a fixed compact set, and hence
		$I_t^+=O(t^{-1/2})$.
		
		As for $I_t^-$, \eqref{eq:phi*} gives
		\begin{equation*}
			I_t^-=e^{-itw_*}\int_{\R^d}\widehat W(-\zeta)
			\frac{e^{it\phi_*(A_t^{-1}\zeta)}}{\omega(\xi^*+A_t^{-1}\zeta)}\,d\zeta.
		\end{equation*}
		On the fixed support of $\widehat W$, Lemma~\ref{lem:limitphase} gives uniform convergence of both phase and amplitude. Then the integrands are uniformly bounded and supported in one compact set. Therefore the dominated convergence theorem gives
		\begin{equation*}
			e^{itw_*}I_t^-\longrightarrow
			C_W:=\frac{1}{w_*}\int_{\R^d}\widehat W(-\zeta)e^{i\Phi_\infty(\zeta)}\,d\zeta.
		\end{equation*}
		By the choice of $\delta$ and~\eqref{eq:detla2}, we have
		$ |\Phi_\infty(\zeta)|\leq\frac{\pi}{4}$ for $\zeta\in B(0,\delta)$.
		Combined with $\supp \widehat W\subset B(0,\delta)$, then
		\begin{equation*}
			\operatorname{Re}C_W
			=\frac{1}{w_*}\int_{\R^d}\widehat W(-\zeta)\cos\Phi_\infty(\zeta)\,d\zeta
			\geq\frac{1}{\sqrt{2}w_*}\int_{\R^d}\widehat W(\zeta)\,d\zeta>0.
		\end{equation*}
		Thus $ C_W\neq0$.
	\end{proof}
	
	\begin{proof}[Proof of the lower bound in Theorem~\ref{thm:main}]
		Fix $d$. 
		By Lemma~\ref{lem:I},
		$$
		S_t=-\frac{e^{-itw_*}}{2i(2\pi)^d} C_W+o(1),
		\quad \text{as}~ t\to+\infty.
		$$
		Since $C_W\neq0$, $|S_t|$ is bounded below by a positive constant $c_d$ for all sufficiently large $t\geq T_d$. By Lemma~\ref{lem:total W},
		\[
		C_{\mathrm{sum}}t^{\sigma_d}\sup_{x\in\Z^d}|G_d(x,t)|\geq|S_t|\geq c_d,\quad t\geq T_d.
		\]
		Since $G_d(x,-t)=-G_d(x,t)$, the result follows.
	\end{proof}

\section*{Declaration on the Use of AI}
	The authors employed AI tools to help explore possible proof strategies, especially in the proof of Lemma \ref{lem:B bound}. 
    All mathematical claims, proofs, and citations were independently checked by the authors, who take full responsibility for any remaining errors.

\end{document}